%% file: project_PDEnonelliptic_arxiv20260830.tex
\documentclass{amsart}
\usepackage[margin=1.3in]{geometry}
\usepackage{color}
\usepackage{float}
\usepackage{amsmath}
\usepackage{amssymb}
\usepackage{amsthm}
\usepackage{mathrsfs}
\usepackage{enumerate}
 
\allowdisplaybreaks

\input{_format}

\numberwithin{equation}{section}

\begin{document}

\title[{Counterexamples for generalizations of non-elliptic Schr\"{o}dinger}]{Counterexamples for generalizations of the non-elliptic Schr\"{o}dinger maximal operator}

\author[Rena Chu]{Rena Chu}
\address{Georg-August-Universität Göttingen, Bunsenstraße 3-5, 37073 Göttingen, Germany}
\email{rena.chu@mathematik.uni-goettingen.de}

\begin{abstract}
For $P=X_1^2+\cdots + X_n^2$,
let $T_t^Pf(x)$ denote the solution to the linear Schr\"{o}dinger equation at time $t$.
In 1980, Carleson asked for the minimal regularity of an initial data function $f\in H^s(\bR^n)$ that guarantees pointwise convergence of $T_t^Pf(x)$ to $f(x)$ as $t\rightarrow 0$. This was resolved by Bourgain \cite{Bou16}, who constructed counterexamples for the Schr\"{o}dinger maximal operator to show that $s\geq n/(2(n+1))$ is necessary, and Du and Zhang \cite{DuZha19}, who proved that $s> n/(2(n+1))$ is sufficient. Rogers, Vargas, and Vega \cite{RVV06} studied the analogous question for the non-elliptic Schr\"{o}dinger maximal operator, where $P = X_1^2-X_2^2 \pm X_3^2\pm \cdots \pm X_n^2$, and proved that, for all $n\geq 2$, $s\geq 1/2$ is necessary and $s>1/2$ is sufficient.
   In this paper, we construct counterexamples for generalizations of the non-elliptic case
   and prove a necessary condition of $s\geq 1/2$ for an infinite class of polynomial symbols $P$.
\end{abstract}

\maketitle

\section{Introduction}
Given a polynomial $P\in \bR[X_1,...,X_n]$, define
    \begin{align*}
        T_t^{P}f(x) = \frac{1}{(2\pi)^n} \int_{\bR^n} \hat{f}(\xi) e^{i(x\cdot \xi + P(\xi)t)}d\xi
    \end{align*}
initially for $f$ of Schwartz class on $\bR^n$. This gives a solution to the initial value problem
    \begin{align}\label{eqn:PDE}
         \begin{cases}
            i\partial_t u - \Pcal(D) u =0, &(x,t)\in \bR^n\times \bR,\\
            u(x,0)=f(x), &x\in \bR^n
        \end{cases}
    \end{align}
where $D=(\partial/\partial x_1,...,\partial/\partial x_n)/i$ and
    \begin{align*}
        \Pcal(D)f(x) = \frac{1}{(2\pi)^n}\int_{\bR^n} e^{ix\cdot \xi} P(\xi)\hat{f}(\xi)d\xi.
    \end{align*}
When $P = X_1^2+\cdots + X_n^2$, \eqref{eqn:PDE} is the linear Schr\"{o}dinger equation, 
for which Carleson \cite{Car80} posed in 1980 the problem of finding the smallest $s>0$ such that
    \begin{align}\label{eqn:pointwise}
        \lim_{t\rightarrow 0} T_t^Pf(x) = f(x), \quad \text{a.e. } x, \text{ for all }f\in H^s(\bR^n).
    \end{align}
In the one-dimensional case, Carleson \cite{Car80} and Dahlberg and Kenig \cite{DahKen82} showed that \eqref{eqn:pointwise} holds if and only if $s\geq 1/4$. For higher dimensions, this was resolved (up to endpoint), after many decades of work (see e.g. \cite{Car80,DahKen82,Cow83,Car85,Sjo87,Veg88,Bou95,MVV96,TaoVar00,Lee06,Bou13,DemGuo16,DGL17,LucRog17,DGLZ18,LucRog19}), 
by Bourgain \cite{Bou16}, whose counterexample construction proved a necessary condition of $s\geq n/(2(n+1))$, and Du and Zhang \cite{DuZha19}, who proved a sufficient condition of $s> n/(2(n+1))$.

In sharp contrast, the non-elliptic Schr\"{o}dinger equation, which corresponds to the symbol $P = X_1^2-X_2^2 \pm X_3^2\pm \cdots \pm X_n^2$, behaves differently: Rogers, Vargas, and Vega \cite{RVV06} showed that for all $n\geq 2$, \eqref{eqn:pointwise} holds for $s>1/2$ and fails for $s<1/2$.

Analogous questions on the pointwise convergence \eqref{eqn:pointwise} arise for other polynomials $P$.
For an arbitrary polynomial $P$ of degree $k\geq 2$, Sj\"{o}lin \cite{Sjo98} showed that convergence fails for $s <1/4$, and Ben-Artzi and Devinatz \cite{BenDev91} and Rogers, Vargas, and Vega \cite{RVV06} proved convergence holds for all $s>1/2$. (We refer the reader to \cite[\S 1.2]{ACP23} for a detailed literature review.) For a long time, convergence in the range $1/4\leq s \leq 1/2$ remained an open question, for polynomials of higher degrees, until the work of An, Chu, and Pierce \cite{ACP23}, who proved counterexamples for the class of diagonal polynomial symbols $P=X_1^k+\cdots + X_n^k$ with $k\geq 3$ and showed that convergence fails for all $s<\frac{1}{4} + \frac{n-1}{4((k-1)n+1)}$.
Subsequently, Eceizabarrena and Ponce-Vanegas \cite{EPV22a} proved a result of the same strength for polynomials $P$ with leading form (homogeneous part of highest degree) $P_k=X_1^k+Q_k$ where $Q_k\in \bQ[X_2,...,X_n]$ is nonsingular and has degree $k$. More recently, Chu and Pierce \cite{CP23} extended this to polynomials with non-diagonal behaviour by defining a new type of rank $r$ to measure the amount of interaction among the variables in a polynomial and proved that convergence fails for $s<\frac{1}{4} + \frac{n-r}{4((k-1)(n-r+1)+1)}$.
In this paper, we study a class of polynomials characterized by vanishing higher-degree terms when the leading form $P_k$ is written in terms of one variable, and prove necessary conditions on $s$ that are independent of dimension and degree.

Given a real symbol $P$, a standard approach to pointwise almost everywhere convergence problems is to study the maximal operator $ f \mapsto \sup_{0<t<1} |T_t^Pf|$.
    To prove the failure of pointwise convergence \eqref{eqn:pointwise} for some $f\in H^s(\bR^n)$, it suffices to prove that the maximal operator is unbounded from $H^s(\bR^n)$ to $L^1_{\mathrm{loc}}(\bR^n)$. (See e.g. \cite[Appendix A]{Pie20} for details of this implication.) Our main result is the following.

 \begin{thm}\label{thm:main}
    Let $n,k\geq 2$. Let $P\in \bR[X_1,...,X_n]$ be a polynomial of degree $k$ with leading form 
  \begin{align}\label{eqn:mainthm_cond}
        P_k=
        Q_k+X_1Q_{k-1}+X_1^dQ_{k-d}+X_1^{d+1}Q_{k-(d+1)}+\cdots +X_1^kQ_0,
    \end{align}
    where $2\leq d \leq k$ and $Q_{k-i}\in \bR[X_2,...,X_n]$ is a form of degree $k-i$ for all $i=0,1,d,...,k$. Assume $Q_{k-1}$ is not identically zero. Define $s_d=(d-1)/2d$ for $2\leq d \leq k$ and $s_\infty=1/2$.
Suppose there exists a constant $C_s$ such that
    \begin{align}\label{eqn:operator_bounded}
        \|\sup_{0<t<1}|T_t^Pf|\|_{L^1(B_n(0,1))} \leq C_s \|f\|_{H^s(\bR^n)}, \quad \text{for all } f\in H^s(\bR^n).
    \end{align}
Then $s\geq s_\infty$ if $Q_{k-i}$ is identically zero for all $d\leq i \leq k$, and otherwise $s\geq s_d$.
\end{thm}
This produces necessary conditions on $s$, beyond $1/4$, that are independent of $n,k$ for all polynomials with nonzero $Q_{k-1}$, as soon as $d\geq 3$. (When $d=2$, this recovers the known condition of $1/4$.) The threshold of $s_d$ improves as $d$ increases; in \S \ref{sec:method}, we shall see that the (nonlinear) lower-degree terms in $X_1$ contribute to more error than the higher-degree terms when approximating $T_t^Pf(x)$, hence imposing more stringent conditions. Moreover, the threshold on $s$ optimizes at $s_\infty$ when all nonlinear terms in $X_1$ vanish.
In particular, for polynomials whose leading forms are of the shape $P_k = Q_k + X_1Q_{k-1}$, where $Q_{k-1}$ is not identically zero, pointwise convergence \eqref{eqn:pointwise} fails for $s<1/2$.  Combined with the sufficient condition of $s>1/2$ given by \cite{BenDev91,RVV06}, this resolves the question of minimal regularity, up to endpoint, for an infinite class of polynomials.

We prove the following, more general version of Theorem \ref{thm:main}.
\begin{thm}\label{thm:main_general}
    Let $n,k\geq 2$. Let $P\in \bR[X_1,...,X_n]$ be a polynomial of degree $k$ with leading form 
  \begin{align*}
        P_k=
        Q_k+X_1Q_{k-1}+X_1^2Q_{k-2}+\cdots +X_1^kQ_0,
    \end{align*}
    where $Q_{k-i}\in \bR[X_2,...,X_n]$ is a form of degree $k-i$ for all $0\leq i\leq k$. Suppose there is a $2\leq d\leq k$ and a proper subset $\Sigma=\{X_{i_1},...,X_{i_m}\}\subsetneq \{X_2,...,X_n\}$ such that $Q_{k-i}\in \bR[X_{i_1},...,X_{i_m}]$ for all $2\leq i \leq d-1$ and such that $Q_{k-1}$ consists of a monomial only in the variables of the complement $\{X_2,...,X_n\}\setminus \Sigma$. Suppose there exists a constant $C_s$ such that \eqref{eqn:operator_bounded} holds. Then $s\geq s_\infty$ if $Q_{k-i}$ is identically zero for all $d\leq i \leq k$, and otherwise $s\geq s_{d}$.
\end{thm}

Note that Theorem \ref{thm:main} follows with $\Sigma$ as the empty set. When $\Sigma$ is nonempty, this provides nontrivial results for leading forms that consist of quadratic terms in $X_1$.
 To see a concrete example of a polynomial for which Theorem \ref{thm:main_general} produces strictly stronger results than Theorem \ref{thm:main},
 consider a ternary cubic with leading form $X_2 X_3^2 + X_1 X_2^2 +X_1^2 X_3$. This satisfies Theorem \ref{thm:main_general} with $d=3$ and $\Sigma=\{X_3\}$, and since $Q_0$ vanishes, Theorem \ref{thm:main_general} gives a necessary condition of $s\geq 1/2$ (while Theorem \ref{thm:main} gives $s\geq 1/4$). As a second, more general example, 
    \begin{align}\label{eqn:ex_general}
        P_k = Q_k(X_2,...,X_n) + X_1(X_2^{k-1} + \tilde{Q}_{k-1}(X_2,...,X_n)) + \sum_{i=2}^{k-1} X_1^i Q_{k-i}(X_3,...,X_n) + X_1^kQ_0,
    \end{align}
where $\tilde{Q}_{k-1}$ is a form of degree $k-1$,
 satisfies Theorem \ref{thm:main_general} with $d=k$ and $\Sigma = \{X_3,...,X_n\}$. If $Q_0$ is identically zero, we deduce $s\geq 1/2$, and otherwise $s\geq (k-1)/2k$.

We describe two more strengths of our main theorem. First, we produce counterexamples for $s$ beyond $1/4$ for forms in $n$ variables whose intertwining rank is $n$ (in the terminology of \cite{CP23}), which had been intractable in previous works. By definition, the intertwining rank of a form $P_k$ is $r=\min_{1\leq i \leq n} r(X_i)$, where $r(X_i)$ is the number of coordinates $X_j$ with which $X_i$ intertwines, where for $j\neq i$, $X_i$ is said to intertwine with $X_j$ if $(\partial^2 / \partial X_i \partial X_j)P_k\not\equiv 0$ and by convention, $X_i$ always interwines with itself. For example, the polynomial $X_2 X_3^2 + X_1 X_2^2 +X_3X_1^2$ from above has intertwining rank $r=n=3$. Furthermore, Chu and Pierce \cite[\S 1.4]{CP23} showed that for a fixed odd $k\geq 3$, the form
    \begin{align}\label{eqn:shape_CP}
        P_k=X_1^k+\cdots + X_n^k +\sum_{2\leq j \leq r} X_1 X_j^{k-1} + \sum_{2\leq i < j \leq n} X_iX_j^{k-1}
    \end{align}
is Dwork-regular (meaning, roughly, nonsingular) over $\mathbb{Q}$ and has intertwining rank $r$. For such polynomials, they proved necessary conditions beyond $1/4$ when $1\leq r<n$ and recovered the known threshold of $1/4$ for $r=n$. In either case, as long as $r\geq 2$, this leading form satisfies the hypotheses of Theorem \ref{thm:main} with $d=k$ and $Q_0$ non-vanishing (in fact, it is of the shape \eqref{eqn:ex_general}), from which we deduce that convergence \eqref{eqn:pointwise} fails for $s<s_k=(k-1)/2k$. This improves the threshold provided by \cite{CP23} for \eqref{eqn:shape_CP} for the case $2\leq r <n$ as soon as $k\geq 4$ and for the case $r=n$ as soon as $k\geq 3$. More generally, for every polynomial to which both \cite[Theorem 1.1]{CP23} and Theorem \ref{thm:main_general} apply, we achieve a better threshold, for any $n\geq 2,k\geq 2$ and $2\leq d \leq k$ in the case where $Q_{k-i}$ vanishes for all $d\leq i \leq k$, and otherwise, for any $n$, as soon as $d\geq 3$ and $k\geq 4$. 
Another interesting feature of our theorem is its application to polynomials whose leading forms have real coefficients rather than rational coefficients, the latter of which was a common hypothesis in the recent works of \cite{ACP23,EPV22a,CP23} due to the need for their arithmetic properties (which, as we shall see in \S \ref{sec:method}, we do not require).

Since the truth of a bound of the form \eqref{eqn:operator_bounded}
is invariant for a polynomial $P$ under a nonsingular, linear change of variables over $\bR$ (see e.g. \cite[\S 3.5]{CP23}), we are able to improve necessary conditions on $s$ for several familiar classes of polynomials, and we state these consequences now.
First, we generalize the result of \cite{RVV06} to higher-degree forms that are diagonal in at least two coordinates and with different signs.

\begin{cor}\label{cor:diagonal_minus}
    Fix $n,k\geq 2$, and let $P\in \bR[X_1,...,X_n]$ be a polynomial of degree $k$ with leading form $P_k= c_1X_1^k-c_2X_2^k +Q$, where $c_1>0,c_2>0$ are real and $Q\in \bR[X_3,...,X_n]$ is a form of degree $k$. Suppose there exists a constant $C_s$ such that \eqref{eqn:operator_bounded} holds. Then $s\geq 1/2$ if $k=2$ and $s\geq 1/3$ if $k\geq 3$.
\end{cor}

Similarly, we improve the result of \cite{ACP23} on odd-degree diagonal forms and extend this to odd-degree forms that are diagonal in at least two coordinates and with the same sign.

\begin{cor}\label{cor:diagonal_plus}
    Fix $n\geq 2 $ and $k\geq 3$ odd, and let $P\in \bR[X_1,...,X_n]$ be a polynomial of degree $k$ with leading form $P_k= c_1X_1^k+c_2X_2^k +Q$, where $c_1>0,c_2>0$ are real and $Q\in \bR[X_3,...,X_n]$ is a form of degree $k$. Suppose there exists a constant $C_s$ such that \eqref{eqn:operator_bounded} holds. Then $s\geq 1/3$.
\end{cor}

In particular, for pure diagonal forms $X_1^k+\cdots + X_n^k$ where $k\geq 3$ is odd, this strengthens the threshold given by \cite{ACP23} for all $n$, as soon as $k\geq 4$.
Furthermore, we show that when $n=2$ and $k\geq 3$ is odd, pointwise convergence fails for $s<1/3$, for ``most'' forms $P_k$. This is motivated by the question of genericity, raised in \cite[\S 1.3]{CP23}; more precisely, what is the behaviour of \eqref{eqn:PDE} when $P_k$ is ``generic''? Let $M_{n,k}$ denote the moduli space of homogeneous polynomials with real coefficients, in $n$ variables and of degree $k$. Recall that a class of forms in $M_{n,k}$ is generic if it is dense in $M_{n,k}$ with respect to the Zariski topology.
 \begin{cor}\label{cor:binary}
    Fix $k\geq 3$ odd, and let $P\in \bR[X_1,X_2]$ be of degree $k$. There exists a generic class $G$ in the moduli space $M_{2,k}$ of binary forms of degree $k$ with real coefficients such that if the leading form $P_k$ belongs in $G$, then, if there exists a constant $C_s$ such that \eqref{eqn:operator_bounded} holds, then $s\geq 1/3$.
\end{cor}

We prove these corollaries in \S \ref{sec:special_cases}.

\subsection{Method of proof}\label{sec:method}
We follow the standard approach of constructing an explicit family $f_j$ of counterexamples that violate the bound \eqref{eqn:operator_bounded}. Thus to show Theorem \ref{thm:main_general}, it suffices to prove the following.

\begin{thm}\label{thm:main2}
    Let $n,k\geq 2$. Let $P\in \bR[X_1,...,X_n]$ be a polynomial of degree $k$ with leading form satisfying the hypotheses of Theorem \ref{thm:main_general}.
  Fix  $s< s_\infty$ if $Q_{k-i}$ is identically zero for all $d\leq i \leq k$, and otherwise, fix $s<s_d$.
Then there exists a sequence of real numbers $R_j\rightarrow \infty$ as $j\rightarrow \infty$ and a sequence of functions $f_j$ such that
    \begin{align*}
        \lim_{j\rightarrow \infty}\frac{\|\sup_{0<t<1} |T_t^Pf_j|\|_{L^1(B_n(0,1))}}{\|f_j\|_{H^s(\bR^n)}} = \infty. 
    \end{align*}
\end{thm}

To prove this, we will construct, for each large $R$, a function $f=f_R$ and a set $\Omega$ of $x$'s such that for each $x\in \Omega$, there is a $t\in (0,1)$ such that
    \begin{align*}
        \frac{|T_t^Pf(x)|}{\|f\|_{H^s(\bR^n)}}
    \end{align*}
tends to infinity as $R$ grows. As motivated in \cite[\S 2]{Pie20}, a natural candidate for $f$ is \[f=\phi(S_1x_1)e(R_1x_1)\prod_{j=2}^n \phi(x_j)e(R_jx_j),\] where $R_j\approx R$ is large, $S_1=R^\sigma$ with $0\leq \sigma \leq 1$, and $\phi$ is a Schwartz function such that $\phi(0)=1$ and whose Fourier transform is supported in $[-1,1]$. With this choice, we compute that $\|f\|_{H^s(\bR^n)}\ll_{\phi,n} R^s S_1^{-1/2}$. Then, upon showing that for each $x\in \Omega$, there exists $t\in (0,1)$ such that $|T_t^Pf(x)|\gg_{\phi,n,P} 1$, possibly subject to conditions on $S_1$ (equivalently, $\sigma$), we arrive at the lower bound $|T_t^Pf(x)| / \|f\|_{H^s(\bR^n)} \gg_{\phi,n,P}R^{\sigma/2 - s}$. This grows to infinity with $R$, as long as $s<\sigma/2$, so for $f$ to violate the upper bound \eqref{eqn:operator_bounded} for $s$ as large as possible, we must maximize $\sigma$.

Now we give a brief overview of our strategy to optimize the value of $\sigma$ and motivate the hypotheses on $P_k$ that appear in our main theorem.
We adapt ideas from the work of Rogers, Vargas, and Vega \cite{RVV06} on the non-elliptic Schr\"{o}dinger equation, where $P=X_1^2-X_2^2\pm X_3^2\pm \cdots \pm X_n^2$. Under a nonsingular, linear change of variables, this can be rewritten as $X_1X_2\pm X_3^2\pm \cdots \pm X_n^2$, where the presence of a linear term in $X_1$ and the absence of higher-degree terms in $X_1$ are particularly advantageous. Indeed, this allows for the scaling parameter $S_1$ of the first coordinate to be large while ensuring error terms (from integration by parts) are controlled. We extend this approach to arbitrary polynomials that have a nonzero linear term in $X_1$ and zero quadratic term in $X_1$ (after possibly renaming variables).
We note that in contrast to the recent work of \cite{ACP23,EPV22a,CP23, EY26} which adapted the key innovation of \cite{Bou16} of introducing exponential sums into the construction of counterexamples (see \cite{Pie20} for a detailed explanation), we do not insert such a sum here in $f$. Instead, we exploit the structure of $P_k$ satisfying the hypotheses of Theorem \ref{thm:main_general} to weaken the condition on $S_1$ which ultimately pushes the threshold of $s$ beyond $1/4$. 

We follow the standard framework of approximating $T_t^Pf$ in terms of $\phi$ by Fourier inversion. To do so, we insert $f$ as defined above into $T_t^Pf$ to get, roughly,
    \begin{align*}
        |T_t^Pf(x)| \approx |\int_{\bR^n} (\prod_{j=1}^n\hat{\phi}(\lambda_j))e((S_1\lam_1,\lam')\cdot(x_1,x') + \nabla P_k(R_1,R')t)e(h(\lambda)t)d\lambda|,
    \end{align*}
where $\lambda=(\lam_1,\lam')$, $\lam'=(\lam_2,...,\lam_n)$, $x'=(x_2,...,x_n)$, $R'=(R_2,...,R_n)$, and
 \begin{align*}
        h(\lam)=\sum_{2\leq |\alpha|\leq k} h_{k,\alpha}(\lam) + \sum_{0\leq i \leq k-1}\sum_{1
        \leq |\alpha|\leq i} h_{i,\alpha}(\lam),
    \end{align*}
where $h_{i,\alpha}(\lam) = \partial^\alpha P_i(R_1,R')(S_1\lam_1,\lam')^\alpha / \alpha!$.
If $h(\lam)t$ is sufficiently small, then by Fourier inversion,
    \begin{align*}
        |T_t^Pf(x)|\approx  |\phi(S_1(x_1+\partial_1 P_k(R_1, R')t))\prod_{j=2}^n \phi(x_j+\partial_j P_k(R_1, R')t) |.
    \end{align*}
For the right-hand side to be ``large'', we require that 
    \begin{align}\label{eqn:cond_t_method}
        t\approx -x_1/\partial_1P_k(R_1,R'),
    \end{align}
(which implicitly demands $\partial_1P_k(R_1,R')$ to be nonzero) and $\partial_j P_k(R_1,R') \ll \partial_1 P_k(R_1,R')$, so that the arguments of $\phi$ are near zero (so that $|T_t^P f(x)| \approx 1$). 
To ensure $h(\lam)t$ is small, we impose additional constraints on $t$, which in turn force restrictions on the size of $S_1$. The stringent cases come from the terms $h_{k,\alpha}(\lam)$ where $\alpha$ is such that $|\alpha|=\alpha_1 \geq 2$. Indeed, for $|\alpha|=\alpha_1 =i$ and for small $\lam$, we have $h_{k,\alpha}(\lam)t \approx (\partial^i P_k / \partial X_1^i)(R_1,R')S_1^i t$. Writing $P_k$ in the form of $ Q_k + X_1 Q_{k-1} + X_1^2Q_{k-2} + \cdots + X_1^k Q_0$, we see that $\partial^i P_k / \partial X_1^i \approx \sum_{j=i}^k X_1^{j-i} Q_{k-j}$. Setting $R_1=0$, we have $h_{k,\alpha}(\lam)t \approx Q_{k-i}S_1^i t.$ If $Q_{k-i}$ is not identically zero, then $ h_{k,\alpha}(\lam)t \approx R^{k-i}S_1^i t \approx R^k (S_1/R)^i t$. Since $S_1\leq R$, this is large when $i$ is small, and when compared with \eqref{eqn:cond_t_method}, this reveals the condition $\sigma \leq (i-1)/i$. In particular, when $i=2$, this is $\sigma \leq 1/2$, which had been the threshold for $\sigma$ in previous works. To break past this barrier, we apply the hypotheses on the structure of $P_k$. We see that the nonlinear, lower-degree terms in $X_1$ (say $X_1^i Q_{k-i}$ for $2\leq i \leq d-1$) contribute to more deviation from the linear term. Thus, assuming these terms vanish identically
allows us to relax the condition on $S_1$, from $\sigma\leq 1/2$ to $\sigma\leq (d-1)/d$.

\subsection{Outline of paper}
In \S \ref{sec:initial_def}, we define the functions $f=f_R$ and give an upper bound on the $H^s$ norm. In \S \ref{sec:compute_T}, we compute the size of $|T_t^Pf|$. In \S \ref{sec:proof}, we prove Theorem \ref{thm:main2}, and in \S \ref{sec:special_cases}, we prove Corollaries \ref{cor:diagonal_minus} to \ref{cor:binary}.

\subsection{Notation}
For a real number $\alpha$, we use the convention $e(\alpha)=e^{i\alpha}$. Correspondingly, we use the normalization $\hat{f}(\xi) = \int_{\bR^r}f(x) e^{-i x\cdot \xi} dx$, and so $f(x) = (2\pi)^{-r}\int_{\bR^r} \hat{f}(\xi) e^{i x \cdot \xi} d\xi$.
The Sobolev space $H^s(\bR^r)$ is defined to be all $f \in \mathcal{S}'(\bR^r)$ with finite Sobolev norm
\[ \|f\|_{H^s(\bR^r)}^2 = \frac{1}{(2\pi)^r} \int_{\bR^r} (1+|\xi|^2)^s |\hat{f}(\xi)|^2 d\xi.\]
We let $A\ll B$ denote $|A|\leq C|B|$ for some constant $C>0$ and $A\ll_\alpha B$ if the constant $C$ depends on $\alpha$.  
For a multi-index $\alpha=(\alpha_1,...,\alpha_r)$, let $|\alpha| = \alpha_1+\cdots + \alpha_r$ and $\alpha! = \alpha_1!\cdots \alpha_r!$, and for a polynomial $F(X_1,...,X_r)$ in $r$ variables, let $\partial^\alpha F = \partial^{|\alpha|}F/(\partial X_1^{\alpha_1}\cdots \partial X_r^{\alpha_r})$.

\section{The initial definition}\label{sec:initial_def}

We begin the construction of the counterexample functions, building on the works of \cite{RVV06,Pie20,CP23}.    
Let $P\in \bR[X_1,...,X_n]$ be a polynomial of degree $k$ with leading form satisfying the hypotheses of Theorem \ref{thm:main_general}, and suppose without loss of generality that $\Sigma=\{X_2,...,X_m\}$ with $m<n$.
Fix parameters $R,S_1\geq 1$ and let $S_1=R^\sigma$ where $0\leq \sigma\leq 1$. 

To guarantee that the value of $t$ we later choose is small enough (more precisely, in \eqref{eqn:cond_t} and \eqref{eqn:cond_partial}), we require the following lower bound on a form evaluated at certain values.

\begin{lemma}\label{lem:coefficients}
  Let $F \in \bR[X_1,...,X_r]$ be a nonzero form of degree $e \geq 1$.
  There exists a tuple $(N_1,...,N_r)\in \bZ^r$, with  $N_i\geq 1$ for all $i$, such that for all $R \geq 1,$
   \[
        |F(N_1R,N_2R,...,N_rR)| \gg_F R^e.
\]
 \end{lemma}
\begin{proof}
    For any $(N_1,...,N_{r})\in \bZ^{r}$, we have $F(N_1R,N_2R,...,N_rR) = R^{e}F(N_1,...,N_{r})$ by homogeneity, so it suffices to show that there exist integral $N_1,...,N_{r}$ such that $F(N_1,...,N_{r}) \neq 0$. For any $B\geq 1$, $|\{(x_1,...,x_r)\in \bZ^r \cap [1,B]^r: F(x_1,...,x_r)=0\}|\ll_e B^{r-1}$ by the Schwartz-Zippel bound (see e.g. \cite[Lemma 10.1]{BCLP22}). Since there are $B^r$ integral points in that box, there exists sufficiently large $B$ such that there is an integral point, say $(N_1,...,N_{r}) \in [1,B]^r$, satisfying
    $F(N_1,...,N_{r})\neq 0.$
\end{proof}

Let $(N_{m+1},...,N_n)$ be as provided by Lemma \ref{lem:coefficients} for the form $Q_{k-1}(X_2,...,X_n)|_{X_i=0, \,i\in \Sigma}$. We check that this form is not identically zero since, by supposition, $Q_{k-1}$ consists of a monomial only in the variables of $\{X_2,...,X_n\}\setminus \Sigma$. 
Fix a Schwartz function $\phi:\bR\rightarrow \bR$ such that $\phi(0)=1$ and $\hat{\phi}$ supported on $[-1,1]$. 
Define
    \begin{align}\label{eqn:def_f}
        f(x) = \phi(S_1x_1) [\prod_{j=2}^m \phi(x_j)][\prod_{j=m+1}^n\phi(x_j)e(N_jRx_j)].
    \end{align}
Then its Fourier transform is
    \begin{align*}
        \hat{f}(\xi)=S_1^{-1}\hat{\phi}(\frac{\xi_1}{S_1})
        [\prod_{j=2}^m \hat{\phi}(\xi_j)][
        \prod_{j=m+1}^n \hat{\phi}(\xi_j-N_jR)].
    \end{align*}

\subsection{An upper bound on the $H^s$ norm} By definition,
    \begin{align*}
       \|f\|_{H^s(\bR^n)}^2 
       &= \frac{1}{(2\pi)^n} \int_{\bR^n} (1+|\xi|^2)^s|\Hat{f}(\xi)|^2 d\xi\\
      & =\frac{1}{(2\pi)^n} \int_{\bR^n} (1+|\xi|^2)^s|S_1^{-1}\hat{\phi}(\frac{\xi_1}{S_1})[\prod_{j=2}^m \hat{\phi}(\xi_j)][
        \prod_{j=m+1}^n \hat{\phi}(\xi_j-N_jR)]|^2 d\xi.   
    \end{align*}
Recall that $\hat{\phi}$ is supported in $[-1,1]$, so
    \begin{align*}       \|f\|_{H^s(\bR^n)}^2&\ll_n R^{2s}\int_{-S_1}^{S_1}\int_{[-1,1]^{m-1}} \int_{[N_jR-1,N_jR+1]^{n-m}} S_1^{-2}|\hat{\phi}(\frac{\xi_1}{S_1})[\prod_{j=2}^m \hat{\phi}(\xi_j)][
        \prod_{j=m+1}^n \hat{\phi}(\xi_j-N_j R)]|^2 d\xi.
    \end{align*}
Let $\lam_1=\xi_1/S_1$, $\lam_j=\xi_j$ for $2\leq j \leq m$, and $\lam_j=\xi_j-N_jR$ for $m+1\leq j \leq n$. Then $d\lam_1 = S_1^{-1}d\xi_1$, and
    \begin{align}\label{eqn:f_Hnorm}
     \|f\|_{H^s(\bR^n)}^2&\ll_n R^{2s} \int_{[-1,1]^n}S_1^{-1}|\prod_{j=1}^n \hat{\phi}(\lam_j)|^2 d\lam   
     \ll_n R^{2s}S_1^{-1}\|\hat{\phi}\|_{L^2(\bR)}^{2n}.
    \end{align}

\section{Bounding $|T_t^Pf(x)|$}\label{sec:compute_T}

Fix $0<c_0<1/2$. Since $\phi$ is smooth and $\phi(0)=1$, there exists a small constant $\delta_0>0$ depending on $c_0,\phi$ such that 
    \begin{align}\label{eqn:phi_delta}
        \phi(y)\geq 1-c_0, \quad  \text{for all } |y|\leq \delta_0.
    \end{align}
 The main result of this section is the following.

\begin{prop}\label{prop:approx_maximal} Let $0<c_0<1/2$ be a small constant and $\delta_0$ be as in \eqref{eqn:phi_delta}. Suppose $\sigma \leq 1$, if $Q_{k-i}$ vanishes for all $d\leq i \leq k$, and otherwise, suppose $\sigma\leq (d-1)/d$. Then there exist constants $0< c_1(\delta_0,P_k),c_2(P_k)<1$ such that for all  $c_1<c_1(\delta_0,P_k),c_2<c_2(P_k)$ and $0<c_3<1$ as small as we like, the following holds. 

Let $R$ be sufficiently large depending on $c_1,c_2,\phi,\sigma,P_k$.
Let $x\in [-c_1,-c_1/2]\times [-c_1,c_1]^{n-1}$ and suppose $t\in (0,1)$ satisfies  
    \begin{align*}
        t &= \frac{-x_1}{\partial_1 P_k(0,R')} + \tau, \qquad \text{where } |\tau| \leq \frac{c_2\delta_0}{S_1R^{k-1}},\qquad \text{and }
        t \leq \frac{c_3}{R^{k-1}},
    \end{align*}
    where $R'=(R_2,...,R_n)$ with $R_j=0$ for $2\leq j \leq m$ and $R_j = N_j R$ for $m+1\leq j \leq n$.
Then $ |T_t^{P}f(x)| \geq (2\pi)^{-n}(1-c_0)^n + E$, where $ E \ll_{\phi,n,P_k}c_3.$
\end{prop}

By a change of variables $\lam_1=\xi_1/S_1$, $\lam_j=\xi_j$ for $2\leq j \leq m$, and $\lam_j=\xi_j-N_jR$ for $m+1\leq j \leq n$,
    \begin{align*}
        T_t^Pf(x)
        =\frac{1}{(2\pi)^n}\int_{\bR^n} \hat{\Phi}(\lam)e(x_1 S_1\lam_1 +\sum_{j=2}^m x_j \lam_j+\sum_{j=m+1}^n x_j(\lam_j+N_jR)
        +P(S_1\lam_1,\lam'+R')t)d\lam,
    \end{align*}
where $\lam = (\lam_1,\lam')$, $\lam' = (\lam_2,...,\lam_n)$, and $\Phi(\lam) = \phi(\lam_1)\cdots \phi(\lam_n)$. Rearranging terms,
    \begin{align}\label{eqn:T_1}
        T_t^Pf(x)= \frac{1}{(2\pi)^n}e(
        \sum_{j=m+1}^n x_jN_jR)\int_{\bR^n} \hat{\Phi}(\lam)e((S_1\lam_1,\lam')\cdot (x_1,x') +P((S_1\lam_1,\lam')+(0,R'))t)d\lam,
    \end{align}
    where $x' = (x_2,...,x_n)$.
To apply Fourier inversion, we remove from the integral higher degree terms in $\lam$. Write $P=P_k+P_{k-1}+\cdots + P_0$, where $P_i$ is homogeneous of degree $i$. By Taylor expansion, 
\begin{align*}
    P((S_1\lam_1,\lam')+(0,R'))
    =P(0,R')+\nabla P_k(0, R')\cdot (S_1\lam_1,\lam') + h(\lam;R',S_1),
\end{align*}
where
\begin{align}\label{eqn:def_h2}
        h(\lam;R',S_1)=\sum_{2\leq |\alpha|\leq k} h_{k,\alpha}(\lam;R',S_1) + \sum_{0\leq i \leq k-1}\sum_{1
        \leq |\alpha|\leq i} h_{i,\alpha}(\lam;R',S_1)
    \end{align}
where $h_{i,\alpha}(\lam;R',S_1) = \partial^\alpha P_i(0,R')(S_1\lam_1,\lam')^\alpha / \alpha!$.
 Then the integral in \eqref{eqn:T_1} is
    \begin{align}\label{eqn:T_2}
      e(P(0,R')t) \int_{[-1,1]^n} \hat{\Phi}(\lam)e((S_1\lam_1,\lam')\cdot((x_1,x')+\nabla P_k(0, R')t)+h(\lam;R',S_1) t)d\lam.
    \end{align}
An application of integration by parts (see e.g. \cite[Lemma 4.4]{CP23}) removes the $e(h(\lam;R',S_1)t)$ term, so that \eqref{eqn:T_2} is
\begin{align}\label{eqn:T_3}
    e(P(0,R')t) e(h(1,...,1; R',S_1)t)\int_{[-1,1]^n} \hat{\Phi}(\lam)e((S_1\lam_1,\lam')\cdot((x_1,x')+\nabla P_k(0,R')t))d\lam +E_1,
\end{align}
where
 \begin{align}\label{eqn:E1_PS}
        E_1\ll_{\phi,n}  \sup_{\kappa\in \{0,1\}^n, |\kappa|\geq 1}\sup_{y\in[-1,1]^n}|\frac{\partial^{|\kappa|}}{\partial y_1^{\kappa_1}\cdots \partial y_n^{\kappa_n}} h(y;R',S_1) t|.
    \end{align}
Upon applying Fourier inversion to the integral in \eqref{eqn:T_3}, we get
    \begin{align*}
        T_t^Pf(x) = M + E,
    \end{align*}
where $ |M| = (2\pi)^{-n} |\phi(S_1(x_1+\partial_1 P_k(0, R')t))\prod_{i=2}^n \phi(x_i+\partial_i P_k(0,R')t)|$
and $ |E| = (2\pi)^{-n}|E_1|.$

We now choose $t$ to give a lower bound on the main term $M$.
Let
    \begin{align}\label{eqn:cond_t}
        t = -\frac{x_1}{\partial_1 P_k(0,R')} + \tau, \qquad |\tau|\leq \frac{c_2\delta_0}{S_1R^{k-1}},
    \end{align}
where, by Lemma \ref{lem:coefficients} and by hypothesis that $Q_{k-1}$ consists of a monomial only in the variables of $\{X_2,...,X_n\} \setminus \Sigma = \{X_{m+1},...,X_n\}$, we have
    \begin{align}\label{eqn:cond_partial}
      \partial_1P_k(0,R') = Q_{k-1}(R')\gg_{P_k} R^{k-1}.
    \end{align}
Choose $0<c_2<1/2$ (depending on $P_k$) sufficiently small so that $c_2 |\partial_i P_k(0,R')| / R^{k-1}<1/2$ for all $1\leq i \leq n$. Let $x\in [-c_1,c_1]^n$, where $0<c_1<1/2$ (depending on $P_k$) satisfies $c_1<\delta_0/4$ and $c_1|\partial_i P_k(0,R')/\partial_1 P_k(0,R')|<\delta_0/4$ for all $2\leq i \leq n$.
Then
    \begin{align*}
       | S_1(x_1+\partial_1 P_k(0,R')t)| = |S_1 \partial_1 P_k(0,R')\tau| \leq c_2\delta_0\frac{| \partial_1 P_k(0,R')|}{R^{k-1}}< \frac{\delta_0}{2},
    \end{align*}
and so by \eqref{eqn:phi_delta}, $   \phi(S_1(x_1+\partial_1 P_k(0,R')t))\geq 1-c_0.$
Similarly, for $2\leq i \leq n$,
    \begin{align*}
        |x_i+\partial_i P_k(0,R')t| \leq |x_i|+| x_1\frac{\partial_i P_k(0,R')}{\partial_1 P_k(0,R')} |+ c_2\delta_0 \frac{|\partial_i P_k(0,R')|}{S_1R^{k-1}}<\delta_0,
    \end{align*}
    so that
    $\phi(x_i+\partial_i P_k(0,R')t)\geq 1-c_0$. Together, $|M| \geq (2\pi)^{-n}(1-c_0)^n.$
To ensure $t\in (0,1)$, suppose without loss of generality that $\partial_1 P_k(0,R')>0$, and let $x_1\in [-c_1,c_1/2]$. Then for sufficiently large $R$ (depending on $c_1,c_2,\phi,\sigma,P_k$), we have $c_1/\partial_1 P_k(0,R')  + c_2/S_1R^{k-1}<1$ and $c_1/2\partial_1 P_k(0,R')  - c_2/S_1R^{k-1}>0$.
    
It remains to bound the error term $E_1$ (and hence $E$). We claim that
    \begin{align}\label{eqn:claim_E}
        E_1\ll_{\phi,n,P_k} R^{k-1}t,
    \end{align}
    subject to the condition $S_1\ll R$ if $Q_{k-i}$ is identically zero for all $d\leq i \leq k$, and $S_1\ll R^{(d-1)/d}$ otherwise.
Recalling \eqref{eqn:def_h2} and \eqref{eqn:E1_PS}, note a priori that for any $0\leq i \leq k$ and $1\leq |\alpha|\leq i$,
\begin{align}\label{eqn:E1_bound}
      \sup_{\kappa\in \{0,1\}^n,|\kappa| \geq 1}\sup_{y\in[-1,1]^n}|\frac{\partial^{|\kappa|}}{\partial y_1^{\kappa_1}\cdots \partial y_n^{\kappa_n}} h_{i,\alpha}(\lam;R',S_1)
        |\ll_{P_k} R^{i-|\alpha|} S_1^{\alpha_1}.
    \end{align}
It suffices to show that this is at most $R^{k-1}$ for the cases $i=k$ with $2\leq |\alpha|\leq k$ and $0\leq i \leq k-1$ with $1 \leq |\alpha|\leq i$.
For $0\leq i \leq k-1$, the right-hand side of \eqref{eqn:E1_bound} is at most $R^{k-1}$ since $S_1\leq R$.
    So suppose $i=k$ and $2\leq |\alpha|\leq k$.
If $\alpha_1=0$, then this is bounded above by $R^{k-2}$. If $|\alpha|>\alpha_1\geq 1$, then this is at most $R^{k-1}(S_1/R)^{\alpha_1}\ll R^{k-1}$.

Lastly, we treat the case $|\alpha|=\alpha_1\geq 2$, for which we exploit the shape of $P_k$. We consider two cases. First suppose $Q_{k-i}$ is identically zero for all $d\leq i \leq k$, that is, $P_k=Q_k+X_1Q_{k-1}+X_1^2Q_{k-2}+\cdots + X_1^{d-1}Q_{k-(d-1)}$. If $\alpha_1 > d-1$, then $\partial^\alpha P_k$ vanishes. Otherwise if $2\leq \alpha_1\leq d-1$, then $\partial^\alpha P_k = \sum_{\alpha_1\leq i \leq d-1} \alpha_1! \binom{i}{\alpha_1}X_1^{i-\alpha_1} Q_{k-i}$. By hypothesis, $Q_{k-i}\in \bR[X_2,...,X_m]$ for $2\leq i \leq d-1$ and $R'=(R_2,...,R_n)$ with $R_j=0$ for $2\leq j \leq m$. Hence $\partial^\alpha P_k (0,R')=0$, so \eqref{eqn:claim_E} is satisfied without additional conditions on $S_1$ other than $S_1\leq R$. 
Now consider the second case where not all $Q_{k-i}$, $d \leq i \leq k$, are identically zero.
If $\alpha_1>d-1$, 
then the right-hand side of \eqref{eqn:E1_bound} is $R^{k-\alpha_1}  S_1^{\alpha_1} = R^k(S_1/R)^{\alpha_1} \ll R^k(S_1/R)^d$,
where $R^k(S_1/R)^d \ll R^{k-1}$ if and only if $S_1^d \ll R^{d-1}$. This forces the condition
    \begin{align}\label{eqn:cond_sigma}
        \sigma\leq (d-1)/d.
    \end{align}
On the other hand, if $2\leq \alpha_1\leq d-1$, then $\partial^\alpha P_k = 
          \sum_{\alpha_1\leq i \leq k} \alpha_1!\binom{i}{\alpha_1}  X_1^{i-\alpha_1} Q_{k-i}$.
By considering cases $\al_1\leq i\leq d-1$ and $d\leq i \leq k$, we deduce $\partial^\alpha P_k(0,R')=0$, which proves the claim \eqref{eqn:claim_E}.

Finally, we impose that $  t\leq c_3/R^{k-1}$
for some $0<c_3<1$ as small as we like. This is compatible with $\eqref{eqn:cond_t}$ by $\eqref{eqn:cond_partial}$. Then $   E\ll_{\phi,n,P_k} R^{k-1}t\leq c_3,$
which concludes the proof of Proposition \ref{prop:approx_maximal}.

\section{Proof of Theorem \ref{thm:main2}}\label{sec:proof}
Let $\Omega = [-c_1,-c_1/2]\times [-c_1,c_1]^{n-1}$. Then by Proposition \ref{prop:approx_maximal}, for all $x\in \Omega$, there exists $t\in (0,1)$ such that $|T_t^Pf(x)|\geq (2\pi)^{-n}(1-c_0)^n  - |E|$ where $E\ll_{\phi,n,P_k}c_3$ for some $c_3$ as small as we like. In particular, choose $c_3$ sufficiently small (depending on $\phi,n,P_k$) so that $|E|\leq  (1/2)(2\pi)^{-n}(1-c_0)^n$. Then $|T_t^Pf(x)|\geq (1/2)(2\pi)^{-n}(1-c_0)^n$. Combining this with the $H^s$ norm of $f$ in \eqref{eqn:f_Hnorm}, we conclude that 
    \begin{align*}
        \frac{\|\sup_{0<t<1} |T_t^Pf|\|_{L^1(B_n(0,1))}}{\|f\|_{H^s(\bR^n)}} \geq \frac{(1-c_0)^n}{2(2\pi)^n}c_1^n R^{-s}S_1^{1/2}\|\hat{\phi}\|^{-n}_{L^2(\bR^n)} \gg_{\phi,n,P_k} R^{\sigma/2-s}.
    \end{align*}
The right-hand side tends to infinity as long as $s<\sigma/2$. To maximize this, choose $\sigma=1$ in the case that $Q_{k-i}$ vanishes for all $d\leq i \leq k$, and otherwise, with respect to the condition \eqref{eqn:cond_sigma}, choose $\sigma = (d-1)/d$.

\section{Special cases}\label{sec:special_cases}

\subsection{Proof of Corollaries \ref{cor:diagonal_minus} and \ref{cor:diagonal_plus}}
It suffices to show that for a binary diagonal form $F=c_1X_1^k+c_2X_2^k$ whose degree is odd if $c_1,c_2$ have the same sign, there exists a change of variables $A\in \GL_2(\bR)$ under which $F$ can be written in the form $Q_k+X_1Q_{k-1}+X_1^3Q_{k-3}+\cdots + X_1^kQ_0$, where $Q_{k-i}\in \bR[X_2]$ has degree $k-i$ and $Q_{k-1}$ is not identically zero. Then, under the change of variables defined by the diagonal block matrix with blocks $A$ and the identity matrix $I_{(n-2)\times (n-2)}$, the form $P_k = c_1X_1^k+c_2X_2^k + Q(X_3,...,X_n)$, in the case $k\geq 3$, can be written in the shape of \eqref{eqn:mainthm_cond} with $d=3$, while in the case $k=2$, $P_k$ can be written in the shape $Q_k+X_1Q_{k-1}$.

\begin{lemma} Fix $k\geq 2$ and let $F = c_1X_1^k+c_2X_2^k$. If $c_1,c_2$ have the same sign, then further assume that $k$ is odd. Then there exists a change of variables $A\in \GL_2(\bR)$ such that $F(AX)$ is of the form $Q_k+X_1Q_{k-1}+X_1^2Q_{k-2}+\cdots + X_1^kQ_0$, where $Q_{k-i}\in \bR[X_2]$ has degree $k-i$ and such that $Q_{k-1}$ is not identically zero and $Q_{k-2}$ is identically zero.
\end{lemma}
    \begin{proof} For $A=(a_{ij})$ where
 $a_{11},a_{12},a_{21},a_{22}$ are real, 
        \begin{align*}
    F(AX) 
    &=\sum_{i=0}^k c'_{k-i}X_1^iX_2^{k-i}, \qquad c'_{k-i} = \binom{k}{i}(c_1 a_{11}^i a_{12}^{k-i}+c_2a_{21}^ia_{22}^{k-i}).
        \end{align*}
        We seek $a_{11},a_{12},a_{21},a_{22}$ such that $ a_{11}a_{22}-a_{12}a_{21}\neq 0$, $c'_{k-1}\neq 0$, and $c'
        _{k-2} =0$.
Note that a priori there is no solution to $c'_{k-2}=0$ if $c_1,c_2$ have the same sign and $k$ is even. So by hypothesis we assume that either $c_1,c_2$ have the same sign and $k$ is odd, or $c_1,c_2$ have different signs. Then the condition $c'_{k-2}=0$ is equivalent to
    \begin{align}\label{eqn:diagonal_a22}
        a_{22}^{k-2}=-c_1 a_{11}^2 a_{12}^{k-2}/c_2a_{21}^2,
    \end{align}
    given $a_{21}\neq 0$. Using this relation and further supposing $a_{12}\neq 0$, we compute that $c_1 a_{11} a_{12}^{k-1}+c_2a_{21} a_{22}^{k-1}=0$ if and only if $c_1a_{11}^k+c_2a_{21}^k=0$, and similarly,
 $a_{11}a_{22}-a_{12}a_{21}=0$ if and only if $c_1a_{11}^k + c_2a_{21}^k= 0$.
    Thus, for any $a_{11}$ and $a_{12}\neq 0$, choose $a_{21}\neq 0$ satisfying $c_1a_{11}^k+c_2a_{21}^k\neq 0$ (so that $\det A\neq 0$ and $c'_{k-1}\neq 0$) and $a_{22}$ satisfying \eqref{eqn:diagonal_a22} (so that $c'_{k-2}=0$).
    \end{proof}

If $k=2$, then the above choices of $A$ ensure that $F(AX)$ does not have nonlinear terms in $X_1$ and so it is of the form $Q_k + X_1Q_{k-1}$ as desired. In the case $k\geq 3$, if in addition we want $Q_{k-3}$ to be identically zero, then there is the additional constraint that $ c_1 a_{11}^3 a_{12}^{k-3}+c_2a_{21}^3 a_{22}^{k-3}=0$.
Combined with \eqref{eqn:diagonal_a22}, we get $c_1a_{11}^k+c_2a_{21}^k=0$ which contradicts $c'_{k-1}\neq 0$.

\subsection{Proof of Corollary \ref{cor:binary}}
Write $F = \sum_{i=0}^k c_i X_1^i X_2^{k-i}$, and let $A=(a_{ij})$ where
 $a_{11},a_{12},a_{21},a_{22}$ are real. Then
    \begin{align*}
        F(AX)
        = \sum_{i=0}^k c'_{k-i} X_1^i X_2^{k-i}, \qquad 
        c'_{k-i} = 
        \sum_{j=0}^{i}\sum_{m=i-j}^{k-j}  c_{m} \binom{m}{i-j} \binom{k-m}{j} a_{11}^{i-j} a_{12}^{m-(i-j)}a_{21}^j a_{22}^{k-m-j}.
    \end{align*}
For ease of computation, let $a_{11}=1, a_{21}=0, a_{22}=1$. Then $\det A =1$ and
    \begin{align*}
    c'_{k-1}=\sum_{m=1}^k c_m m a_{12}^{m-1}, \qquad c'_{k-2} = \sum_{m=2}^k c_m \binom{m}{2}a_{12}^{m-2}
    \end{align*}
are single-variable polynomials in $a_{12}$.
We seek $\tilde{a}_{21}$ such that $c'_{k-1}(\tilde{a}_{21})\neq 0$ and $c'_{k-2}(\tilde{a}_{21})=0$. 
Let $R$ denote the resultant of $c'_{k-2}, c'_{k-1}$. By definition, $R$ is a polynomial in the coefficients of $c'_{k-2}, c'_{k-1}$ such that $R$ vanishes if and only if $c'_{k-2}, c'_{k-1}$ have a common root. First, we check that $R$ is not the zero polynomial. Consider the case where $c_1=1, c_k=1$ and all other $c_i=0$. Then $c'_{k-1} = ka_{12}^{k-1} +1$ and $c'_{k-2} = \binom{k}{2}a_{12}^{k-2}$, so they do not share a common root, and hence $R$ is not identically zero. Let $F$ correspond to a point lying in the intersection of $R\neq 0$ and $c_k\neq 0$ in the moduli space of binary forms of degree $k$. Then, since $c_k\neq 0$, $c'_{k-2}$ is a polynomial in $a_{12}$ of odd degree and hence has a real root $\Tilde{a}_{12}$. On the other hand, since $R\neq 0$, $c'_{k-1}$ does not share a root with $c'_{k-2}$ and hence $c'_{k-1}(\Tilde{a}_{12})\neq 0$. Thus for binary forms of odd degree $k$ that lie in the Zariski-open set defined by the intersection of $R\neq 0$ and $c_k\neq 0$, there exists a nonsingular, linear change of variables $A$ such that $F(AX) = Q_k+X_1Q_{k-1}+X_1^2Q_{k-2}+\cdots + X_1^kQ_0$ with $Q_{k-1}\neq 0$ and $Q_{k-2}=0$.

\section*{Acknowledgements}
 The author thanks Lillian B. Pierce for her continued encouragement and for many helpful discussions. 
The author was partially supported by NSF DMS-2200470 and the Katherine Goodman Stern Fellowship from The Graduate School at Duke University for portions of this project.

\bibliographystyle{alpha}
\bibliography{_bibliography}

\end{document}

%% file: _format.tex
\numberwithin{equation}{section}

\newtheorem{thm}{Theorem}[section]
\newtheorem*{thm*}{Theorem}
\newtheorem{prop}[thm]{Proposition}
\newtheorem{lemma}[thm]{Lemma}
\newtheorem*{lemma*}{Lemma}

\newtheorem{cor}[thm]{Corollary}

\theoremstyle{definition}

\theoremstyle{remark}

\definecolor{pink}{rgb}{1,.2,.6}
\definecolor{orange}{rgb}{0.7,0.3,0}
\definecolor{blue}{rgb}{.2,.6,.75}
\definecolor{green}{rgb}{.4,.7,.4}
\definecolor{purple}{RGB}{127,0,255}

\newcommand{\bQ}{\mathbb{Q}}
\newcommand{\bR}{\mathbb{R}}

\newcommand{\bZ}{\mathbb{Z}}

\newcommand{\GL}{\mathrm{GL}}

\newcommand{\al}{\alpha}

\newcommand{\lam}{\lambda}

\newcommand{\Pcal}{\mathcal{P}}

\newcommand{\beq}{\begin{equation}}
\newcommand{\eeq}{\end{equation}}

\newcommand{\ba}{\begin{align}}
\newcommand{\ea}{\end{align}}

\makeatletter
\def\@tocline#1#2#3#4#5#6#7{\relax
  \ifnum #1>\c@tocdepth 
  \else
    \par \addpenalty\@secpenalty\addvspace{#2}%
    \begingroup \hyphenpenalty\@M
    \@ifempty{#4}{%
      \@tempdima\csname r@tocindent\number#1\endcsname\relax
    }{%
      \@tempdima#4\relax
    }%
    \parindent\z@ \leftskip#3\relax \advance\leftskip\@tempdima\relax
    \rightskip\@pnumwidth plus4em \parfillskip-\@pnumwidth
    #5\leavevmode\hskip-\@tempdima
      \ifcase #1
       \or\or \hskip 1em \or \hskip 2em \else \hskip 3em \fi%
      #6\nobreak\relax
    \hfill\hbox to\@pnumwidth{\@tocpagenum{#7}}\par
    \nobreak
    \endgroup
  \fi}
\makeatother

\newcommand{\genlegendre}[4]{%
  \genfrac{(}{)}{}{#1}{#3}{#4}%
  \if\relax\detokenize{#2}\relax\else_{\!#2}\fi
}

\makeatletter
\let\@@pmod\pmod
\DeclareRobustCommand{\pmod}{\@ifstar\@pmods\@@pmod}
\def\@pmods#1{\mkern4mu({\operator@font mod}\mkern 6mu#1)}
\makeatother